\documentclass[leqno,11pt]{amsart}

\usepackage{amsfonts}
\usepackage{latexsym,amssymb,amsthm,amsmath,amscd}
\usepackage{latexsym,amssymb,amsthm,amsmath}
\usepackage{graphics,amscd}
\usepackage{amssymb}
\usepackage{eucal}
\usepackage{amsmath}

\theoremstyle{plain}
\newtheorem{theorem}{Theorem}[section]

\newtheorem{corollary}{Corollary}[section]

\numberwithin{equation}{section}

\theoremstyle{remark}
\newtheorem{remark}{Remark}[section]
 \numberwithin{equation}{section}

\def\<{\left < }
\def\>{\right >}
\def\({\left ( }
\def\){\right )}

\begin{document}
\title[Conformal canonical vector field]{Euclidean submanifolds with
conformal canonical vector field}
\author{ Bang-Yen Chen}
\author[S. Deshmukh]{Sharief Deshmukh}

\begin{abstract}
The position vector field $\hbox{\bf x}$ is the most elementary and natural
geometric object on a Euclidean submanifold $M$. The position vector field
plays very important roles in mathematics as well as in physics. Similarly,
the tangential component $\hbox{\bf x}^T$ of the position vector field is
the most natural vector field tangent to the Euclidean submanifold $M$. We
simply call the vector field $\hbox{\bf x}^T$ the \textit{canonical vector
field} of the Euclidean submanifold $M$.

In earlier articles \cite{C16,C17a,CV17,CW17}, we investigated Euclidean
submanifolds whose canonical vector fields are concurrent, concircular, or
torse-forming. In this article we study Euclidean submanifolds with
conformal canonical vector field. In particular, we characterize such
submanifolds. Several applications are also given. In the last section we
present three global results on complete Euclidean submanifolds with
conformal canonical vector field.
\end{abstract}

\maketitle

\section{Introduction}

For an $n$-dimensional submanifold $M$ in the Euclidean $m$-space $\mathbb{E}%
^m$, the most elementary and natural geometric object is the position vector
field $\mathbf{x}$ of $M$. The position vector is a Euclidean vector $%
\mathbf{x} =\overrightarrow{OP}$ that represents the position of a point $%
P\in M$ in relation to an arbitrary reference origin $O\in \mathbb{E}^m$.

The position vector field plays important roles in physics, in particular in
mechanics. For instance, in any equation of motion, the position vector $\mathbf{x}(t)$ is usually the most sought-after quantity because the
position vector field defines the motion of a particle (i.e., a point mass):
its location relative to a given coordinate system at some time variable $t$. The first and the second derivatives of the position vector field with
respect to time $t$ give the velocity and acceleration of the particle.

For a Euclidean submanifold $M$ of $\mathbb{E}^m$, there exists a natural
decomposition of the position vector field $\mathbf{x}$ given by: 
\begin{align}
\mathbf{x}=\mathbf{x}^T+\mathbf{x}^N,
\end{align}
where $\mathbf{x}^T$ and $\mathbf{x}^N$ are the tangential and the normal
components of $\mathbf{x}$, respectively. We denote by $|{\hbox{\bf x}}^T|$
and $|\hbox{\bf x}^N |$ the lengths of $\hbox{\bf x}^T$ and of $\hbox{\bf x}%
^N$, respectively.

A vector field $v$ on a Riemannian manifold $N$ is called a \textit{%
torse-forming vector field} if it satisfies (cf. \cite{S54,Y40,Y44}) 
\begin{align}  \label{2.7}
\nabla_X v=\varphi X+\alpha(X) v,\;\; \forall X\in TN,
\end{align}
for some function $\varphi$ and 1-form $\alpha$ on $M$, where $\nabla$
denotes the Levi-Civita connection of $M$. In the case that $\alpha$ is
identically zero, $v$ is called a \textit{concircular vector field}. In
particular, if $\alpha=0$ and $\varphi=1$, then $v$ is called a \textit{%
concurrent vector field}.

In earlier articles, we have investigated Euclidean submanifolds whose
canonical vector fields are concurrent \cite{C16,C17a}, concircular \cite%
{CW17}, or torse-forming \cite{CV17}. See \cite{C17c,C17d} for two recent
surveys on several topics in differential geometry associated with position
vector fields on Euclidean submanifolds.

A tangent vector field $v$ on a Riemannian manifold $(N,g)$ is called a {\it conformal vector field} if it satisfies
\begin{align} {\mathcal L}_vg= 2\varphi g,\end{align}
where $\mathcal L$ denotes the Lie derivative of $(N,g)$ and $\varphi$ is called the {\it potential function} of $v$. 

In this article we study Euclidean submanifolds with conformal canonical
vector field. In particular, we characterize such submanifolds. Several
applications are also given. In the last section we present three global
results on complete Euclidean submanifolds with conformal canonical vector
field.

\section{Preliminaries}

\label{section 2}

Let $x: M\to \mathbb{E}^m$ be an isometric immersion of a connected
Riemannian manifold $M$ into a Euclidean $m$-space $\mathbb{E}^m$. For each
point $p\in M$, we denote by $T_pM$ and $T^\perp_p M$ the tangent space and
the normal space of $M$ at $p$, respectively.

Let $\nabla$ and $\tilde\nabla$ denote the Levi--Civita connections of $M$
and ${\mathbb{E}}^{m}$, respectively. The formulas of Gauss and Weingarten
are given respectively by (cf. \cite{book73,book11,book17}) 
\begin{align}  \label{2.1}
&\tilde \nabla_XY=\nabla_X Y+h(X,Y), \\
& \tilde\nabla_X\xi =-A_\xi X+D_X\xi ,  \label{2.2}
\end{align}
for vector fields $X,\,Y$ tangent to $M$ and $\xi$ normal to $M$, where $h$
is the second fundamental form, $D$ the normal connection and $A$ the shape
operator of $M$.

For each normal vector $\xi$ at $p$, the shape operator $A_\xi$ is a
self-adjoint endomorphism of $T_pM$. The second fundamental form $h$ and the
shape operator $A$ are related by 
\begin{equation}  \label{2.3}
g(A_\xi X,Y)g=\tilde g(h(X,Y), \xi),
\end{equation}
where $g$ and $\tilde g$ denote the metric of $M$ and the metric of the
ambient Euclidean space, respectively.

The mean curvature vector $H$ of an $n$-dimensional submanifold $M$ is
defined by 
\begin{align}  \label{2.4}
H=\frac{1}{n} \mathrm{trace}\; h.
\end{align}

A submanifold $M$ is called \textit{totally umbilical} (respectively, 
\textit{totally geodesic}) if its second fundamental form $h$ satisfies 
\begin{align}  \label{2.5}
h(X,Y)=g(X,Y) H
\end{align}
identically (respectively, $h=0$ identically).

A submanifold is said to be \textit{umbilical with respect to a normal
vector field $\xi$} if its second fundamental form $h$ satisfies 
\begin{align}  \label{2.6}
\tilde g(h(X,Y), \xi)=\mu g(X,Y)
\end{align}
for some function $\mu$. In particular, a submanifold $M$ is called \textit{%
pseudo-umbilical} if it is umbilical with respect to the mean curvature
vector field $H$ of $M$.

The Laplace operator $\Delta $ of $M$ acting on smooth vector fields on a
Riemannian $n$-manifold $(M,g)$ is defined by%
\begin{align}  \label{2.7}
\Delta X=\sum\limits_{i=1}^{n}\left( \nabla _{e_{i}}\nabla _{e_{i}}X-\nabla
_{\nabla _{e_{i}}e_{i}}X\right),
\end{align}
where $\{e_1,\ldots,e_n\}$ is an orthonormal local frame of $M$.

\section{Euclidean submanifolds with conformal canonical vector field}

The following result characterizes all Euclidean submanifolds with conformal
canonical vector field.

\begin{theorem}
\label{T:3.1} Let $M$ be a submanifold of the Euclidean $m$-space $\mathbb{E}%
^m$. Then the canonical vector field $\hbox{\bf x}^T$ of $M$ is a conformal
vector field if and only if $M$ is umbilical with respect to the normal
component $\hbox{\bf x}^N$ of the position vector field $\hbox{\bf x}$.
\end{theorem}

\begin{proof}
Let $M$ be a submanifold of $\mathbb{E}^{m}$. Then, by using the fact that
the position vector field is a concurrent vector, we derive from Gauss' and
Weingarten's formulas that 
\begin{align}
& Z =\tilde \nabla_Z \hbox{\bf x}=\nabla_Z \hbox{\bf x}^T+h(\hbox{\bf x}%
^T,Z)-A_{\hbox{\bf x}^N} Z+D_Z \hbox{\bf x}^N  \notag
\end{align}
for any vector $Z$ tangent to $M$, where $\tilde \nabla$ and $\nabla$ are
the Levi-Civita connections of $\mathbb{E}^{n+1}$ and of $M$, respectively.
By comparing the tangential and normal components of the last equation, we
obtain 
\begin{align}  \label{3.1}
& \nabla_Z \hbox{\bf x}^T=Z+A_{\hbox{\bf x}^N} Z, \\
& h(\hbox{\bf x}^T,Z)= - D_Z \hbox{\bf x}^N.  \label{3.2}
\end{align}

On the other hand, it is well-known that the Lie derivative on $M$ satisfies
(see, e.g. \cite[Page 18]{book17} or \cite{Y57}) 
\begin{align}  \label{3.3}
({\mathcal{L}}_v g)(X,Y)=g(\nabla_X v,Y)+g(X,\nabla_Y v)
\end{align}
for any vector fields $X,Y,v$ tangent to $M$.

After combining \eqref{3.1} and \eqref{3.3} we find 
\begin{align}  \label{3.4}
({\mathcal{L}}_{\hbox{\bf x}^T} g)(X,Y)=2 g(X,Y)+ g(A_{\hbox{\bf x}^N} X
,Y)+g(X,A_{\hbox{\bf x}^N} Y)
\end{align}
Therefore, by applying \eqref{2.3} we obtain 
\begin{align}  \label{3.5}
({\mathcal{L}}_{\hbox{\bf x}^T} g)(X,Y)=2 g(X,Y)+ 2 g(h(X,Y), \hbox{\bf x}^N)
\end{align}
for vector fields $X,Y$ tangent to $M$.

Now, let us suppose that the canonical vector field $\hbox{\bf x}^T$ of the
submanifold $M$ is a conformal vector field. Then we have 
\begin{align}  \label{3.6}
{\mathcal{L}}_{\hbox{\bf x}^T} g=2\varphi g
\end{align}
for a function $\varphi$.

From \eqref{3.5} and \eqref{3.6} we derive 
\begin{align}  \label{3.7}
g(h(X,Y), \hbox{\bf x}^N)=(\varphi-1) g(X,Y),
\end{align}
which shows that $M$ is umbilical with respect to the normal component $%
\hbox{\bf x}^N$ of the position vector field $\hbox{\bf x}$.

Conversely, let us assume that the submanifold $M$ is umbilical with respect
to the normal component $\hbox{\bf x}^N$ so that we have 
\begin{align}  \label{3.8}
g(h(X,Y), \hbox{\bf x}^N)=\eta g(X,Y)
\end{align}
for some function $\eta$. Then it follows from \eqref{3.5} and \eqref{3.8}
that 
\begin{align}  \label{3.9}
({\mathcal{L}}_{\hbox{\bf x}^T} g)(X,Y)=2(\eta+2) g(X,Y).
\end{align}
Thus the canonical vector field $\hbox{\bf x}^T$ is a conformal vector field
on $M$.
\end{proof}

\begin{remark}
By applying the same proof as Theorem \ref{T:3.1}, we also know that Theorem \ref
{T:3.1} remains true for space-like submanifolds of pseudo-Euclidean spaces.
\end{remark}

A unit normal vector field $\xi$ of a Euclidean submanifold $M$ is called a 
\textit{parallel} (resp., \textit{nonparallel}\/) normal vector field if $
D\xi=0$ (resp., $D\xi\ne 0$) everywhere on $M$ (cf. \cite{book73,CY71,CY73}).

An easy consequence of Theorem \ref{T:3.1} is the following.

\begin{corollary}
\label{C:3.1} Let $M$ be a submanifold of $\mathbb{E}^m$ with conformal
canonical vector field. If $\hbox{\bf x}^N\ne 0$ and $\hbox{\bf x}^N/|
\hbox{\bf x}^N|$ is a parallel normal vector field, then either

\begin{itemize}
\item[\textrm{(1)}] $M$ lies in a hyperplane $\mathbb{E}^{m-1}$ of $\mathbb{E}^{m}$, or

\item[\textrm{(2)}] $M$ lies in hypersphere of $S^{m-1}$ centered the origin
of $\mathbb{E}^{m}$.
\end{itemize}
\end{corollary}

\begin{proof}
Let $M$ be a submanifold of $\mathbb{E}^m$ with conformal canonical vector
field. If $\hbox{\bf x}^N\ne 0$ and $\hbox{\bf x}^N/|\hbox{\bf x}^N|$ is a
parallel normal vector field, then it follows from Theorem \ref{T:3.1} that $
M$ is umbilical with respect to the parallel unit normal vector field $
\hbox{\bf x}^N/|\hbox{\bf x}^N|$ satisfying \eqref{3.8}.

If $\eta$ in \eqref{3.8} vanishes identically, then it is easy to verify
that $M$ lies in a hyperplane $\mathbb{E}^{m-1}$ of $\mathbb{E}^{m}$.

On the other hand, if $\eta\ne 0$, then it follows from \cite[Theorem 3.3]
{CY71} that $M$ lies in hypersphere of $S^{m-1}$ centered the origin of $
\mathbb{E}^{m}$.
\end{proof}

In the case that $M$ is a Euclidean hypersurface of $\mathbb{E}^{n+1}$ we
have:

\begin{corollary}
\label{C:3.2} Let $M$ be a hypersurface of $\mathbb{E}^{n+1}$ with conformal
canonical vector field. If $\hbox{\bf x}^N\ne 0$, then either

\begin{itemize}
\item[\textrm{(1)}] $M$ lies in hypersphere of $S^{n}$ centered the origin
of $\mathbb{E}^{n+1}$ or

\item[\textrm{(2)}] $M$ lies in a hyperplane which does not contained the
origin of $\mathbb{E}^{n+1}$.
\end{itemize}
\end{corollary}

\begin{proof}
Let $M$ be a hypersurface of $\mathbb{E}^{n+1}$. Suppose that the canonical
vector field $\hbox{\bf x}^T$ of $M$ is a conformal vector field. If $%
\hbox{\bf x}^N\ne 0$, then the unit normal vector field of $M$ is a parallel
normal vector field automatically. Hence Theorem \ref{T:3.1} implies that $M$
lies either in a hypersphere of $S^{n}$ centered the origin of $\mathbb{E}%
^{n+1}$ or in a hyperplane of $\mathbb{E}^{n+1}$.

If the second case occurs, then the hyperplane does not contained the origin
of $\mathbb{E}^{n+1}$; otherwise one has $\hbox{\bf x}^N=0$ which is a
contradiction.
\end{proof}

For Euclidean submanifolds of codimension 2, we have the following.

\begin{corollary}
\label{C:3.3} Let $(M,g)$ be an $n$-dimensional submanifold of $\mathbb{E}%
^{n+2}$ with $n>3$ and $\hbox{\bf x}^N\ne 0$. If the canonical vector field $%
\hbox{\bf x}^T$ of $M$ is a conformal vector field, then we have:

\begin{itemize}
\item[\textrm{(1)}] If $\frac{\hbox{\bf x}^N}{|\hbox{\bf x}^N|}$ is a
parallel normal section, then $(M,g)$ lies either a hyperplane or in a
hypersphere of $\mathbb{E}^{n+2}$.

\item[\textrm{(2)}] If $\frac{\hbox{\bf x}^N}{|\hbox{\bf x}^N|}$ is a
nonparallel normal section, then $(M,g)$ is a conformally flat space.
Moreover, in this case $M$ is the locus of $(n-1)$-spheres.
\end{itemize}
\end{corollary}

\begin{proof}
Let $(M,g)$ be an $n$-dimensional submanifold of $\mathbb{E}^{n+2}$ with $%
n>3 $ and $\hbox{\bf x}^N\ne 0$. If the canonical vector field $\hbox{\bf x}%
^T$ of $M$ is a conformal vector field, it follows from Theorem \ref{T:3.1}
that $M$ is umbilical with respect the normal direction $\hbox{\bf x}^N$.

If ${\hbox{\bf x}^N}/{|\hbox{\bf x}^N|}$ is a parallel normal section,
Corollary \ref{C:3.1} implies that $M$ lies in a hyperplane or in a
hypersphere of $\mathbb{E}^{n+2}$.

If ${\hbox{\bf x}^N}/{|\hbox{\bf x}^N|}$ is a nonparallel normal section, it
follows from \cite[Theorem 3]{CY73} that $(M,g)$ is a conformally flat
space. Moreover, in this case it also follows from \cite[Theorem 4]{CY73}
that the submanifold is a locus of $(n-1)$-spheres in $\mathbb{E}^{n+1}$.
\end{proof}

\section{Application to Yamabe solitons}

The Yamabe flow was introduced by R. Hamilton at the same time as the Ricci
flow (cf. \cite{H98}). It deforms a given manifold by evolving its metric
according to 
\begin{align}  \label{4.1}
\frac{\partial}{\partial t}g(t)=-R(t) g(t),
\end{align}
where $R(t)$ denotes the scalar curvature of the metric $g(t)$. Yamabe
solitons correspond to self-similar solutions of the Yamabe flow.

A Riemannian manifold $(M, g)$ is called a \textit{Yamabe soliton} if it
admits a vector field $v$ such that 
\begin{equation}  \label{4.2}
\frac{1}{2}{\mathcal{L}}_{v}g= (R-\lambda) g,
\end{equation}
where $\lambda$ is a real number. The vector field $v$ as in the definition
is called a \textit{soliton vector field} for $(M, g)$. We denote the Yamabe
soliton satisfying \eqref{4.2} by $(M,g,v,\lambda)$.

By applying Theorem \ref{T:3.1} we have the following.

\begin{corollary}
If a Euclidean submanifold $(M,g)$ of $\mathbb{E}^{m}$ is a Yamabe soliton
with the canonical vector field $\hbox{\bf x}^T$ as its soliton vector
field, then $\hbox{\bf x}^T$ is a conformal vector field.
\end{corollary}

\begin{proof}
Assume that Euclidean submanifold $(M,g)$ of the Euclidean $m$-space $%
\mathbb{E}^{m}$ is a Yamabe soliton with its canonical vector field $%
\hbox{\bf x}^T$ as the soliton vector field. Then it follows from \cite[%
Theorem 3.1]{CD17} that the second fundamental form $h$ of $M$ satisfies 
\begin{align}  \label{4.3}
\tilde g(h(V,W),\hbox{\bf x}^N) =(R-\lambda-1)g(V,W)
\end{align}
for vectors $V,W$ tangent to $M$, where $R$ is the scalar curvature of $M$
and $\lambda$ is a constant. Hence $M$ is umbilical with respect to $%
\hbox{\bf x}^N$. Consequently, the canonical vector field $\hbox{\bf x}^T$
is a conformal vector field of $M$ according to Theorem \ref{T:3.1}.
\end{proof}

\section{Application to generalized self-similar submanifolds}

Consider the mean curvature flow for an isometric immersion $x:M\to \mathbb{E%
}^m$, that is, consider a one-parameter family $x_t=x(\,\cdot\,, t)$ of
immersions $x_t: M\to \mathbb{E}^{m}$ such that 
\begin{align}  \label{5.1}
\frac{d}{dt}x(p,t)=H(p,t),\;\; x(p,0)=x(p),\;\; p\in M.
\end{align}
is satisfied, where $H(p,t)$ is the mean curvature vector of $M_t$ in $%
\mathbb{E}^m$ at $x(p,t)$.

An important class of solutions to the mean curvature flow equations are 
\textit{self-similar shrinkers} which satisfy a system of quasi-linear
elliptic PDEs of the second order, namely, 
\begin{equation}  \label{5.2}
H=-\hbox{\bf x}^{N},
\end{equation}
where $\hbox{\bf x}^{N}$ is the normal component of the position vector
field of $x:M\rightarrow \mathbb{E}^{m}$ as before. Self-shrinkers play an
important role in the study of the mean curvature flow because they describe
all possible blow up at a given singularity of a mean curvature flow.

In view of \eqref{5.2}, we simply call a Euclidean submanifold $M$ a \textit{%
generalized self-similar submanifold} if it satisfies 
\begin{equation}  \label{5.3}
\hbox{\bf x}^{N}=f H
\end{equation}
for some function $f$.

Obviously, it follows from \eqref{5.3} that every Euclidean hypersurface is
a generalized self-similar hypersurface automatically.

By applying Theorem \ref{T:3.1} we have the following.

\begin{corollary}
Let $M$ be a generalized self-similar submanifold of the Euclidean $m$-space 
$\mathbb{E}^m$. Then the canonical vector field of $M$ is a conformal vector
field if and only if $M$ is a pseudo-umbilical submanifold.
\end{corollary}

\begin{proof}
Let $M$ be a generalized self-similar submanifold of $\mathbb{E}^m$. Then we
have \eqref{5.3}. If the canonical vector field of $M$ is a conformal vector
field, then \eqref{3.7} holds for some function $\varphi$. Clearly, it
follows from \eqref{3.7} and \eqref{5.3} that $M$ is pseudo-umbilical.

Conversely, if $M$ is pseudo-umbilical, then \eqref{5.3} implies that $M$ is
umbilical with respect $\hbox{\bf x}^N$. Hence Theorem \ref{3.1} implies
that the canonical vector field of $M$ is a conformal vector field.
\end{proof}

\section{Three global results on complete submanifolds with conformal canonical vector field}

Recall that Euclidean submanifolds in this article are assumed to be
connected (see Preliminaries). In this article, by a \textit{complete
submanifold} of $\mathbb{E}^m$ we mean a complete Riemannian manifold
isometrically immersed in $\mathbb{E}^m$.

\begin{theorem}
\label{T:6.1} Suppose that the canonical vector field $\mathbf{x}^{T}$ on a
complete submanifold $M$ of $\,\mathbb{E}^{m}$ is non-parallel and
conformal. If $\mathbf{x}^{T}$ satisfies 
\begin{equation*}
\Delta \mathbf{x}^{T}=-\lambda \mathbf{x}^{T}
\end{equation*}%
for a non-negative constant $\lambda $, then either $M$ is isometric to an $%
n $-sphere ${S}^{n}(c)$ or to the Euclidean space $\mathbb{E}^{n}$ with $%
n=\dim M$.
\end{theorem}

\begin{proof}
Suppose that the canonical vector field $\mathbf{x}^{T}$ is a non-parallel,
conformal vector field satisfying 
\begin{align}  \label{6.2}
{\mathcal{L}}_{\hbox{\bf x}^T} g=2\varphi g
\end{align}
for some function $\varphi$. Using equation \eqref{3.1}, we compute the
curvature tensor of the submanifold as 
\begin{equation*}
R(X,Y)\mathbf{x}^{T}=\left( \nabla A_{\mathbf{x}^{N}}\right) (X,Y)-\left(
\nabla A_{\mathbf{x}^{N}}\right) (Y,X)\text{,}
\end{equation*}%
where the covariant derivative 
\begin{equation*}
\left( \nabla A_{\mathbf{x}^{N}}\right) (X,Y)=\nabla _{X}A_{\mathbf{x}%
^{N}}Y-A_{\mathbf{x}^{N}}\nabla _{X}Y.
\end{equation*}
Using \eqref{6.2} and equation \eqref{3.7}, we compute 
\begin{equation*}
R(X,Y)\mathbf{x}^{T}=(X\varphi )Y-(Y\varphi )X\text{.}
\end{equation*}%
Taking inner product with $\mathbf{x}^{T}$ in above equation, we get 
\begin{equation*}
(X\varphi )g(Y,\mathbf{x}^{T})=(Y\varphi )g(X,\mathbf{x}^{T})\text{,}
\end{equation*}%
that is, $(X\varphi )\mathbf{x}^{T}=g(X,\mathbf{x}^{T})\nabla \varphi $,
where $\nabla \varphi $ is the gradient of the function $\varphi $. The last
relation shows that the vector fields $\nabla \varphi $ and $\mathbf{x}^{T}$
are parallel. Hence there exists a smooth function $\beta $ such that%
\begin{equation}
\nabla \varphi =\beta \mathbf{x}^{T}\text{.}  \label{6.3}
\end{equation}%
Now, using \eqref{3.1}, we compute%
\begin{equation*}
\Delta \mathbf{x}^{T}=\sum\limits_{i}\left( \nabla A_{\mathbf{x}^{N}}\right)
(e_{i},e_{i})\text{,}
\end{equation*}%
which in view of equation \eqref{3.7} gives $\Delta \mathbf{x}^{T}=\nabla
\varphi $, which in view of \eqref{6.3}, yields $\Delta \mathbf{x}^{T}=\beta 
\mathbf{x}^{T}$. Using the condition in the statement, we get $\beta
=-\lambda $. Thus equation \eqref{6.3} gives%
\begin{equation}
\nabla \varphi =-\lambda \mathbf{x}^{T}\text{,}  \label{6.4}
\end{equation}%
which in view of \eqref{3.1}, gives%
\begin{equation}
\nabla _{X}\nabla \varphi =-\lambda \left( X+A_{\mathbf{x}^{N}}X\right)
=-\lambda \varphi X\text{,}  \label{6.5}
\end{equation}%
where we have used \eqref{3.8}. If $\varphi $ is not a constant, then
equation \eqref{6.4} insures that $\lambda $ is a positive constant (since $%
\mathbf{x}^{T}\neq {0}$ being a non-parallel vector). Thus, equation %
\eqref{6.5} is Obata's differential equation, which
proves that $M$ is isometric to ${S}^{n}(\sqrt{\lambda })$ (cf. \cite{Obata}).

If $\varphi $ is a constant, then the function%
\begin{equation*}
f=\frac{1}{2}|\mathbf{x}^{T}|^{2}\text{,}
\end{equation*}%
on using equations \eqref{3.1} and \eqref{3.7} gives%
\begin{equation*}
Xf=g(X+A_{\mathbf{x}^{N}}X,\mathbf{x}^{T})=\varphi g(X,\mathbf{x}^{T})\text{,%
}
\end{equation*}%
that is, the gradient $\nabla f$ is given by%
\begin{equation}
\nabla f=\varphi \mathbf{x}^{T}\text{.}  \label{6.6}
\end{equation}%
Hence, the Hessian $H_{f}$ of the function $f$ is given by%
\begin{equation}
H_{f}(X,Y)=\varphi ^{2}g(X,Y)\text{.}  \label{6.7}
\end{equation}%
Note that if $f$ is a constant function, equation \eqref{6.6} would imply
either the constant $\varphi =0$ or $\mathbf{x}^{T}={0}$, and both in view
of equations \eqref{3.1} and \eqref{3.7} will imply that $\mathbf{x}^{T}$ is
a parallel vector field, which is contrary to our assumption in the
hypothesis. Hence $f$ is a non-constant function that satisfies equation %
\eqref{6.7} for nonzero constant $\varphi ^{2}$ implies that $M$ is
isometric to the Euclidean space $\mathbb{E}^{n}$ (cf. \cite[Theorem 1]{P11}).
\end{proof}

Next, we use the potential function $\varphi $ of the conformal canonical
vector field $\mathbf{x}^{T}$ and the support function $f$ in the definition %
\eqref{5.3} of generalized self-similar submanifold in proving the next
result.

\begin{theorem}
\label{T:6.2} Let $M$ be a generalized self-similar complete submanifold of
the Euclidean space $\mathbb{E}^{m}$. If the canonical vector field $\mathbf{%
x}^{T}$ is a conformal vector field satisfying 
\begin{equation*}
Ric(\mathbf{x}^{T},\mathbf{x}^{T})+\frac{n}{2}\left[ \mathbf{x}^{T}\!\varphi
+| H|^{2} (\mathbf{x}^{T}\! f)\right] \geq 0\text{,}\;\; n=\dim M,
\end{equation*}
then either $M$ is isometric to the Euclidean $n$-space $\mathbb{E}^{n}$ or
it is a submanifold of constant mean curvature of a hypersphere ${S}%
^{m-1}(c) $ of $\mathbb{E}^{m}$.
\end{theorem}

\begin{proof}
Equation \eqref{3.7} gives $ng(H,\mathbf{x}^{N})=n(\varphi -1)$, which in
view of equation \eqref{5.3} yields
\begin{equation}
\varphi =1+f| H|^{2}\text{.}  \label{6.8}
\end{equation}%
Taking covariant derivative in equation (5.3) and using \eqref{3.2}, we get 
\begin{equation}
-h(X,\mathbf{x}^{T})=(Xf)H+fD_{X}H\text{.}  \label{6.9}
\end{equation}%
Now, using equations \eqref{6.8} and \eqref{6.9}, we have 
\begin{eqnarray}
X\varphi &=&(Xf)|H|^{2}+2fg\left( D_{X}H,H\right)  \notag \\
&=&-(Xf)|H|^{2}-2g(H,h(X,\mathbf{x}^{T}))\text{.}  \label{6.10}
\end{eqnarray}

Recall that the expression for Ricci tensor of a submanifold derived from
Gauss' equation gives 
\begin{equation}
Ric(\mathbf{x}^{T},\mathbf{x}^{T})=ng(H,h(\mathbf{x}^{T},\mathbf{x}
^{T}))-\sum\limits_{i}\left\Vert h(e_{i},\mathbf{x}^{T})\right\Vert ^{2} 
\text{,}  \label{6.11}
\end{equation}%
where $\{e_{1},e_{2},...,e_{n}\}$ is a local orthonormal frame on $M$.

Inserting equation \eqref{6.11} in equation \eqref{6.10} gives%
\begin{equation*}
\mathbf{x}^{T}\varphi +|H|^{2}(\mathbf{x}^{T}f) +\frac{2}{n}Ric(\mathbf{x}%
^{T},\mathbf{x}^{T})=-\frac{2}{n}\sum\limits_{i}\left\Vert h(e_{i},\mathbf{x}%
^{T})\right\Vert ^{2}\text{,}
\end{equation*}%
which in view of the condition in the hypothesis gives%
\begin{equation}
h(X,\mathbf{x}^{T})=0,  \label{6.12}
\end{equation}
for $X$ tangent to $M$. Now, using equation \eqref{6.12}, we find 
\begin{equation}
Ric(X,\mathbf{x}^{T})=0\text{.}  \label{6.13}
\end{equation}%
However, using \eqref{3.1} and \eqref{3.7}, we have $R(X,Y)\mathbf{x}%
^{T}=(X\varphi )Y-(Y\varphi )X$, which gives 
\begin{align}  \label{6.14}
Ric(Y,\mathbf{x}^{T})=-(n-1)(Y\varphi ).
\end{align}
Thus, in view of equation \eqref{6.13}, $\varphi $ is a constant. Therefore
equation \eqref{6.10} implies that
\begin{equation}
(Xf)| H|^{2}=0\text{.}  \label{6.15}
\end{equation}

Now, define a function 
\begin{equation*}
F=\frac{1}{2}|\mathbf{x}^{T}|^{2}\text{,}
\end{equation*}
which has gradient $\nabla F=\varphi \mathbf{x}^{T}$ and Hessian
\begin{equation*}
H_{F}(X,Y)=\varphi ^{2}g(X,Y)\text{.}
\end{equation*}

If $F$ is not a constant, then as $\nabla F=\varphi \mathbf{x}^{T}$,
constant $\varphi ^{2}$ is nonzero, then $M$ is isometric to the Euclidean
space $\mathbb{E}^{n}$ (cf. \cite{P11}).

If $F$ is a constant, then $|\mathbf{x}^{T}|$ is constant and equations 
\eqref{3.2} and \eqref{6.12} give $| \mathbf{x}^{N}|$ is constant.
Consequently, $|\mathbf{x}|$ is a constant and this proves $M$ is a
submanifold of the hypersphere ${S}^{m-1}(c)$. Now, equation \eqref{6.15}
gives $(Xf)|H|^{2}=0$, so either $H=\mathbf{0}$ or $f$ is a constant.

Now, we claim that $f$ is a nonzero constant, for if $f=0$, then equation 
\eqref{5.3} will give $\mathbf{x}^{N}=0$, which by equation \eqref{3.1}
implies $\nabla _{X}\mathbf{x}^{T}=X$, and as $|\mathbf{x}^{T}|$ is a
constant, we get $g(X,\mathbf{x}^{T})=0$ for any smooth vector field $X$
tangent to $M$, that is, $\mathbf{x}^{T}=0$, that is, $\mathbf{x}={0}$ and
it is a contradiction. Therefore $f$ is a nonzero constant. Consequently,
equation \eqref{5.3} implies that $|H|$ is constant.
\end{proof}

Recall that a normal vector field $\xi$ to a Euclidean submanifold $M$ is said to be {\it parallel along a smooth curve} $\gamma :I\rightarrow M$ if
$D_{\gamma ^{{\prime }}} \xi \equiv 0.$ Also, a smooth function $f:M\rightarrow R$ is {\it constant along} $\gamma $ if $\gamma ^{{\prime }}f\equiv 0$.

For a totally geodesic $n$-space $\mathbb{E}^{n}$ of $\mathbb{E}^{m}$, it is
known that the canonical vector field $\mathbf{x}^{T}$ is a concurrent
vector field satisfying 
\begin{equation}\label{6.15}
\nabla _{X}\mathbf{x}^{T}=X.
\end{equation}
Hence the canonical vector field $\hbox{\bf x}^{T}$ is a non-parallel vector
field. Also, it follows from \eqref{3.3} and \eqref{6.15} that $\mathcal{L}_{
\hbox{\bf x}^{T}}g=2g$. Thus the canonical vector field $\hbox{\bf x}^{T}$ is a conformal vector field
with constant potential $\varphi =1$. Furthermore, the mean curvature vector
field $H$ of $\mathbb{E}^{n}$ is zero vector which is trivially a parallel normal vector field. 

Conversely, we prove the following.

\begin{theorem}
\label{T:6.3} Let $M$ be a complete submanifold of $\mathbb{E}^{m}$ whose
canonical vector field $\hbox{\bf x}^{T}$ is non-parallel and conformal. If
the potential function $\varphi $ of $\hbox{\bf x}^{T}$ is constant along
the integral curves of $\hbox{\bf x}^{T}$ and the mean curvature vector
field $H$ of $M$ is parallel along the integral curves of $\hbox{\bf
x}^{T}$, then $M$ is isometric to a Euclidean space.
\end{theorem}

\begin{proof}
Suppose that the potential function $\varphi $ of $\hbox{\bf x}^{T}$ is
constant along the integral curves of $\hbox{\bf x}^{T}$ and that the mean
curvature vector field $H$ of $M$ is parallel along the integral curves of $%
\hbox{\bf x}^{T}$. Then we have $\mathbf{x}^{T}\varphi =0$ and $D_{\mathbf{x}%
^{T}}H=0$. Then, by applying \eqref{6.14}, we get 
\begin{equation}
Ric(\mathbf{x}^{T},\mathbf{x}^{T})=-(n-1)\mathbf{x}^{T}\varphi =0\text{.}
\label{6.16}
\end{equation}

Also, equation \eqref{3.7} implies $g(H,\mathbf{x}^{N})=\varphi -1$, which,
in view of the fact that $H$ is parallel along the integral curves of $%
\mathbf{x}^{T}$, the equation \eqref{3.2} gives%
\begin{equation*}
\mathbf{x}^{T}\varphi =-g(H,h(\mathbf{x}^{T},\mathbf{x}^{T})).
\end{equation*}%
Since $\varphi $ is constant along integral curves of $\mathbf{x}^{T}$,
equation \eqref{6.16} and above equation yield 
\begin{equation}
g(H,h(\mathbf{x}^{T},\mathbf{x}^{T}))=0\text{.}  \label{6.17}
\end{equation}%
Using equations \eqref{6.16} and \eqref{6.17} in equation \eqref{6.11} gives 
\begin{equation*}
h(X,\mathbf{x}^{T})=0\text{,}
\end{equation*}%
for any $X$ tangent to $M$. The above equation implies $Ric(X,\mathbf{x}%
^{T})=0$ for $X$ tangent to $M$, which proves $X\varphi =0$. Hence $\varphi $
is a constant.

Now, define the function $f=\frac{1}{2}\left\vert \mathbf{x}^{T}\right\vert
^{2}\text{,}$ which in view of \eqref{3.1} and \eqref{3.7} gives the
gradient $\nabla f$ and the Hessian of $f$ as 
\begin{equation}
\nabla f=\varphi \mathbf{x}^{T}\text{,\quad }H_{f}(X,Y)=\varphi ^{2}g(X,Y)%
\text{.}  \label{6.18}
\end{equation}

If $f$ is constant, then  \eqref{6.18} implies either constant $\varphi
=0 $ or $\mathbf{x}^{T}=0$ and both of these in view of equations \eqref{3.1}%
, \eqref{3.7} will imply that $\mathbf{x}^{T}$ is a parallel vector field
which is contrary to our assumption. Hence $f$ must be a nonconstant
function satisfying the Hessian condition in \eqref{6.18} with nonzero constant $\varphi $.
Consequently, $M$ is isometric to a Euclidean space (cf. \cite[Theorem 1]%
{P11}).
\end{proof}

\begin{remark}
For further global results on compact Euclidean submanifolds with conformal
canonical vector fields, see \cite{AAD17,DA17}.
\end{remark}

\medskip

\noindent \textbf{Acknowledgements:} This work is supported by King Saud
University, Deanship of Scientific Research, College of Science Research
Center.

\vskip.25in

\noindent \textsl{Bang-Yen Chen }

\textsl{\noindent 2231 Tamarack Drive, }

\textsl{\noindent Okemos, Michigan 48864, U.S.A.}

\noindent e-mail: {chenb@msu.edu}

\vskip.2in

\noindent \textsl{Sharief Deshmukh }

\textsl{\noindent Department of Mathematics, }

\textsl{\noindent King Saud University }

\textsl{\noindent P.O. Box-2455 Riyadh-11451, Saudi Arabia}

\noindent e-mail:{shariefd@ksu.edu.sa}


\end{document}